\documentclass{article}
\usepackage{fancyhdr, hyperref}
\usepackage[T1]{fontenc}
\usepackage{listings}
\usepackage{authblk}
\usepackage{natbib}
\usepackage{bm,amsmath}
\usepackage{amsfonts}
\usepackage[english]{babel}
\usepackage{amsthm}
\newtheorem{lemma}{Lemma}
\lstdefinelanguage{maple}%
  {morekeywords={and,assuming,break,by,catch,description,do,done,elif,else,end,error,export,fi,finally,for,from,global,if,implies,in,intersect,local,minus,mod,module,next,not,od,option,options,or,proc,quit,read,return,save,stop,subset,then,to,try,union,until,use,uses,while,xor},%
   morendkeywords={algebraic,anyfunc,anything,atomic,boolean,complex,constant,cx_infinity,cx_zero,embedded_axis,embedded_imaginary,embedded_real,equation,even,extended_numeric,extended_rational,finite,float,fraction,function,identical,imaginary,indexable,indexed,integer,list,literal,module,moduledefinition,name,neg_infinity,negative,negint,negzero,nonnegative,nonnegint,nonposint,nonpositive,nonreal,numeric,odd,polynom,pos_infinity,posint,positive,poszero,procedure,protected,radical,range,rational,ratpoly,real_infinity,realcons,relation,sequential,set,sfloat,specfunc,string,symbol,tabular,uneval,zppoly,ASSERT,Array,ArrayOptions,CopySign,DEBUG,Default0,DefaultOverflow,DefaultUnderflow,ERROR,EqualEntries,EqualStructure,FromInert,Im,MPFloat,MorrBrilCull,NameSpace,NextAfter,Normalizer,NumericClass,NumericEvent,NumericEventHandler,NumericStatus,Object,OrderedNE,RETURN,Re,Record,SDMPolynom,SFloatExponent,SFloatMantissa,Scale10,Scale2,SearchText,TRACE,ToInert,Unordered,UpdateSource,_hackwareToPointer,_jvm,_local,_maplet,_savelib,_treeMatch,_unify,_xml,abs,add,addressof,anames,andmap,appendto,array,assemble,assign,assigned,attributes,bind,call_external,callback,cat,coeff,coeffs,conjugate,convert,crinterp,debugopts,define_external,degree,denom,diff,disassemble,divide,dlclose,done,entries,eval,evalb,evalf,evalgf1,evalhf,evalindets,evaln,expand,exports,factorial,frem,frontend,gc,genpoly,gmp_isprime,goto,has,hastype,hfarray,icontent,ifelse,igcd,ilog10,ilog2,implies,indets,indices,inner,iolib,iquo,irem,is_gmp,isqrt,kernelopts,lcoeff,ldegree,length,lexorder,lhs,localGridInterfaceRun,lowerbound,lprint,macro,map,map2,max,maxnorm,member,membertype,min,minus,mod,modp,modp1,modp2,mods,mul,mvMultiply,negate,nops,normal,numboccur,numelems,numer,op,order,ormap,overload,parse,piecewise,pointto,print,print_preprocess,readlib,reduce_opr,remove,rhs,rtable,rtableInfo,rtable_convolution,rtable_eval,rtable_histogram,rtable_indfns,rtable_is_zero,rtable_normalize_index,rtable_num_dims,rtable_num_elems,rtable_options,rtable_redim,rtable_scale,rtable_scanblock,rtable_size,rtable_sort_indices,rtable_zip,savelib,searchtext,select,selectremove,seq,series,setattribute,sign,sort,ssystem,stop,streamcall,subs,subset,subsindets,subsop,substring,system,table,taylor,tcoeff,time,timelimit,traperror,trunc,type,typematch,unames,unbind,upperbound,userinfo,wbOpen,wbOpenURI,writeto,~Array,~Matrix,~Vector,args,nargs,procname,RootOf,Float,thismodule,thisproc,_options,_noptions,_rest,_nrest,_params,_nparams,_passed,_npassed,_nresults,static,Catalan,true,false,FAIL,infinity,Pi,gamma,integrate,libname,NULL,Order,printlevel,lasterror,lastexception,Digits,constants,undefined,I,UseHardwareFloats,Testzero,Normalizer,NumericEventHandlers,Rounding,Catalan,FAIL,Pi,false,gamma,infinity,true,ansi,echo,errorbreak,errorcursor,indentamount,labeling,labelwidth,patchlevel,plotdevice,plotoptions,plotoutput,postplot,preplot,prettyprint,printbytes,prompt,quiet,screenheight,screenwidth,showassumed,verboseproc,version,warnlevel,ASSERT,bytesalloc,bytesused,cputime,dagtag,gcbytesavail,gcbytesreturned,gctimes,maxdigits,maximmediate,memusage,printbytes,profile,system,version,wordsize,_Inert,And,Non,Not,Or,SERIES,SymbolicInfinity,TEXT,algebraic,algext,algfun,algnum,algnumext,anyfunc,anything,arctrig,atomic,boolean,complex,complexcons,constant,cubic,cx_infinity,cx_zero,embedded_axis,embedded_imaginary,embedded_real,equation,even,evenfunc,expanded,extended_numeric,extended_rational,facint,finite,float,fraction,function,hfloat,identical,imaginary,indexable,indexed,infinity,integer,laurent,linear,list,listlist,literal,logical,mathfunc,matrix,moduledefinition,monomial,name,neg_infinity,negative,negint,negzero,nonnegative,nonnegint,nonposint,nonpositive,nonreal,nothing,numeric,odd,oddfunc,package,point,polynom,pos_infinity,posint,positive,poszero,prime,protected,quadratic,quartic,radext,radfun,radfunext,radical,radnum,radnumext,range,rational,ratpoly,real_infinity,realcons,relation,scalar,sequential,set,sfloat,specfunc,specindex,sqrt,stack,string,symbol,symmfunc,tabular,trig,truefalse,truefalseFAIL,undefined,uneval,vector,zppoly,AFactor,AFactors,AiryAi,AiryAiZeros,AiryBi,AiryBiZeros,AngerJ,ArrayDims,ArrayElems,ArrayIndFns,ArrayNumDims,Berlekamp,BesselI,BesselJ,BesselJZeros,BesselK,BesselY,BesselYZeros,Beta,Cache,ChebyshevT,ChebyshevU,CheckArgs,Chi,Ci,Complex,ComplexRange,Content,CoulombF,CylinderD,CylinderU,CylinderV,D,DESol,Describe,Det,Diff,Dirac,DistDeg,Divide,Ei,EllipticCE,EllipticCK,EllipticE,EllipticF,EllipticK,EllipticPi,Eval,Expand,Explore,ExportVector,Factor,Factors,Fraction,FresnelC,FresnelS,Fresnelf,Fresnelg,GAMMA,GF,Gausselim,Gaussjord,Gcd,Gcdex,GegenbauerC,HFloat,HankelH1,HankelH2,Heaviside,Hermite,HermiteH,HeunB,HeunBPrime,HeunC,HeunCPrime,HeunD,HeunDPrime,HeunG,HeunGPrime,HeunT,HeunTPrime,ImportVector,Indep,Int,Intat,Interp,InverseJacobiAM,InverseJacobiCD,InverseJacobiCN,InverseJacobiCS,InverseJacobiDC,InverseJacobiDN,InverseJacobiDS,InverseJacobiNC,InverseJacobiND,InverseJacobiNS,InverseJacobiSC,InverseJacobiSD,InverseJacobiSN,Irreduc,IsMatrixShape,IsVectorShape,IsWorksheetInterface,JacobiAM,JacobiCD,JacobiCN,JacobiCS,JacobiDC,JacobiDN,JacobiDS,JacobiNC,JacobiND,JacobiNS,JacobiP,JacobiSC,JacobiSD,JacobiSN,JacobiTheta1,JacobiTheta2,JacobiTheta3,JacobiTheta4,KelvinBei,KelvinBer,KelvinHei,KelvinHer,KelvinKei,KelvinKer,KummerM,KummerU,LaguerreL,Lcm,LegendreP,LegendreQ,Limit,LommelS1,LommelS2,MOLS,MathieuA,MathieuB,MathieuC,MathieuCE,MathieuCEPrime,MathieuCPrime,MathieuExponent,MathieuFloquet,MathieuFloquetPrime,MathieuS,MathieuSE,MathieuSEPrime,MathieuSPrime,Matrix,MatrixOptions,MeijerG,Normal,Nullspace,Power,Powmod,Prem,Primfield,Primitive,Product,Psi,Quo,RESol,Randpoly,Ratrecon,RealRange,Rem,Resultant,Roots,Shi,Si,Smith,Sqrfree,Ssi,Stirling1,Stirling2,String,StruveH,StruveL,Sum,TopologicalSort,Trace,Vector,WARNING,WeberE,WeierstrassP,WeierstrassPPrime,WeierstrassSigma,WeierstrassZeta,WhittakerM,WhittakerW,Wrightomega,about,addcoords,additionally,addproperty,algsubs,alias,allvalues,andseq,apply,applyop,applyrule,arccos,arccosh,arccot,arccoth,arccsc,arccsch,arcsec,arcsech,arcsin,arcsinh,arctan,arctanh,assume,asympt,bernoulli,bernstein,binomial,branches,ceil,charfcn,chrem,coeftayl,collect,combine,comparray,compiletable,compoly,content,convergs,copy,cos,cosh,cot,coth,coulditbe,csc,csch,dataplot,define,definemore,depends,dilog,dinterp,discont,discrim,dismantle,dsolve,eliminate,ellipsoid,erf,erfc,erfi,euler,eulermac,evala,evalapply,evalc,evalr,evalrC,example,exists,exp,extrema,factor,factors,fdiscont,fixdiv,floor,fnormal,forall,forget,frac,freeze,fsolve,galois,gcd,gcdex,getassumptions,hasassumptions,hasfun,hasoption,help,history,identify,ifactor,ifactors,igcdex,ilcm,ilog,implicitdiff,index,info,initialcondition,insertpattern,int,intat,interp,intsolve,invfunc,invztrans,iperfpow,iratrecon,iroot,irreduc,is,iscont,isolate,isolve,ispoly,isprime,isqrfree,issqr,ithprime,latex,lcm,leadterm,limit,ln,lnGAMMA,log,log10,log2,maptype,match,maximize,minimize,modpol,msolve,mtaylor,nextprime,norm,nprintf,odetest,orseq,packages,patmatch,plot,plot3d,plotsetup,poisson,polylog,powmod,prem,prevprime,primpart,printf,product,proot,protect,psqrt,quo,radfield,radnormal,rand,randomize,randpoly,rationalize,ratrecon,readdata,readstat,realroot,redefine,reduce,related,rem,residue,resultant,root,rootbound,roots,round,rsolve,rtable_dims,rtable_elems,scanf,sec,sech,selectfun,shake,showtime,signum,simplify,sin,singular,sinh,sinterp,smartplot,smartplot3d,solve,sprem,sprintf,sqrfree,sscanf,sturm,sturmseq,subtype,sum,surd,symmdiff,tablelook,tan,tanh,testeq,thaw,trigsubs,unapply,unassign,undefine,unprotect,unwindK,usage,value,verify,version,whattype,xormap,xorseq,ztrans,interface,readline,with,unwith}%
  otherkeywords={\%,\%\%,\%\%\%,\$define,\$elif,\$else,\$endif,\$file,\$ifdef,\$ifndef,\$include,\$undef},%
  sensitive=true,%
  morestring=[b]",%
  morestring=[b]`,%
  morecomment=[l]\#,%
  morecomment=[s]{(*}{*)}%
  }[keywords,comments,strings]
\newcommand{\diag}[1]{\ensuremath{\mathrm{diag}{#1}}}

\newcommand{\scof}[1]{\ensuremath{y_{#1}}}

\newcommand{\complex}{\mathbb{C}}
\newcommand{\fieldF}{\mathbb{F}}

\usepackage{xfrac}
\usepackage{bm}

\newcommand{\mat}[1]{\ensuremath{\boldsymbol{#1}}}

\newcommand{\A}{\ensuremath{\mat{A}}}
\newcommand{\B}{\ensuremath{\mat{B}}}

\newcommand{\C}{\ensuremath{\mat{C}}}
\newcommand{\D}{\ensuremath{\mat{D}}}
\newcommand{\E}{\ensuremath{\mat{E}}}
\newcommand{\F}{\ensuremath{\mat{F}}}

\newcommand{\G}{\ensuremath{\mat{G}}}
\renewcommand{\H}{\ensuremath{\mat{H}}}
\newcommand{\I}{\ensuremath{\mat{I}}}

\renewcommand{\L}{\ensuremath{\mat{L}}}

\renewcommand{\P}{\ensuremath{\mat{P}}}

\newcommand{\R}{\ensuremath{\mat{R}}}

\renewcommand{\S}{\ensuremath{\mat{S}}}

\newcommand{\U}{\ensuremath{\mat{U}}}
\newcommand{\V}{\ensuremath{\mat{V}}}
\newcommand{\W}{\ensuremath{\mat{W}}}

\newcommand{\Y}{\ensuremath{\mat{Y}}}
\newcommand{\Z}{\ensuremath{\mat{Z}}}

\renewcommand{\emph}[1]{\textsl{#1}}
\newtheorem{theorem}{Theorem}[section]

\newtheorem{acks}[theorem]{Acknowledgements}

\providecommand{\keywords}[1]
{
  \small	
  \textbf{\textit{Keywords---}} #1
}

\providecommand{\amsclassification}[1]
{
  \small	
  \textbf{\textit{2010 Mathematics Subject Classification---}} #1
}

\usepackage{xcolor}
\definecolor{darkgreen}{rgb}{0,.35,0}
\definecolor{darkblue}{rgb}{0,0,.5}
\definecolor{darkred}{rgb}{.6,0,0}

\hypersetup{bookmarksnumbered = true,colorlinks,linkcolor = darkblue,citecolor=darkgreen, urlcolor = darkred}

\begin{document}

\title{Equivalences for Linearizations of Matrix Polynomials}
\author[1]{Robert M.~Corless\footnote{rcorless@uwaterloo.ca}}
\author[1]{Leili Rafiee Sevyeri\footnote{leili.rafiee.sevyeri@uwaterloo.ca}}
\author[2]{B.~David Saunders\footnote{saunders@udel.edu}}
\affil[1]{David R.~Cheriton School of Computer Science\\ University of Waterloo
,Canada}
\affil[2]{Department of Computer and Information Sciences University of Delaware}
\date{}   

\maketitle
\begin{abstract}
One useful standard method to compute eigenvalues of matrix polynomials $\P(z) \in \mathbb{C}^{n\times n}[z]$ of degree at most~$\ell$ in~$z$ (denoted \emph{of grade $\ell$}, for short) is to first transform $\P(z)$ to an equivalent \emph{linear} matrix polynomial $\L(z)=z\B-\A$, called a companion pencil, where $\A$ and $\B$ are usually of larger dimension than $\P(z)$ but $\L(z)$ is now only of grade~$1$ in $z$. The eigenvalues and eigenvectors of $\L(z)$ can be computed numerically by, for instance, the QZ algorithm.
The eigenvectors of $\P(z)$, including those for infinite eigenvalues, can also be recovered from eigenvectors of $\L(z)$ if $\L(z)$ is what is called a ``strong linearization'' of $\P(z)$.  
In this paper we show how to use algorithms for computing the Hermite Normal Form of a companion matrix for a scalar polynomial to direct the discovery of unimodular matrix polynomial \emph{cofactors} $\E(z)$ and $\F(z)$ which, via the equation $\E(z)\L(z)\F(z) = \diag( \P(z), \I_n, \ldots, \I_n)$, explicitly show the equivalence of $\P(z)$ and $\L(z)$.
By this method we give new explicit constructions for several linearizations using different polynomial bases. We contrast these new unimodular pairs with those constructed by strict equivalence, some of which are also new to this paper. We discuss the limitations of this experimental, computational discovery method of finding unimodular cofactors.  
\end{abstract}

%






\amsclassification{65F15, 15A22, 65D05}


\section{Introduction}
Given a field $\fieldF$ and a set of polynomials $\phi_k(z) \in \fieldF[z]$ for $0 \le k \le \ell$ that define a basis for polynomials of grade $\ell$ (``grade'' is short for ``degree at most'') then a square \emph{matrix polynomial} $\P(z)$ is an element of $\fieldF^{n\times n}[x]$ which we can write as
\begin{displaymath}
\P(z) = \sum_{k=0}^\ell \A_k \phi_k(z)\>.
\end{displaymath}
The matrix coefficients $\A_k$ are elements of $\fieldF^{n\times n}$. For concrete exposition, take $\fieldF$ to be the field of complex numbers $\complex$. 
The case when the ``leading coefficient'' is singular---meaning the matrix coefficient of $z^\ell$, once the polynomial is expressed in the monomial basis---can also be of special interest.  
Normally, only with the case of \emph{regular} square matrix polynomials, for which there exists some $z^* \in \fieldF$ with $\det \P(z^*) \ne 0$ is considered.  Although our intermediate results require certain nonsingularity conditions, our results using strict equivalence are ultimately valid even in the case when the determinant of $\P(z)$ is identically zero.  

Matrix polynomials are of significant classical and current interest: see the surveys~\cite{guttel2017nonlinear} and~\cite{mackey2015polynomial} for more information about their theory and applications.  In this present paper we use Maple to make small example computations in order to discover and prove certain facts about one method for finding eigenvalues of matrix polynomials, namely \emph{linearization}, which means finding an equivalent grade $1$ matrix polynomial $\L(z)$ when given a higher-grade matrix polynomial $\P(z)$ to start.

The paper~\cite{amiraslani2008linearization} was the first to systematically study matrix polynomials in alternative polynomial bases.  See~\cite{BuenoCachadina2020} for a more up-to-date treatment.
\section{Definitions and Notation}
\label{sect:defs}
A \emph{companion pencil} $\L(z)=z\B-\A$ for a matrix polynomial $\P(z)$ has the property that $\det\L(z) = \alpha\det\P(z)$ for some nonzero $\alpha \in \fieldF$.  This means that the eigenvalues of the companion pencil are the eigenvalues of the matrix polynomial.

A \emph{linearization} $\L(z)$ of a matrix polynomial $\P(z)$ has a stronger requirement: a linearization is a pencil $\L(z) = z\B - \A$ which is \emph{equivalent} to $\P(z)$ in the following sense: there exist unimodular\footnote{In this context, the matrix polynomial $\A(z)$ is \emph{unimodular} if and only if $\det\A(z) \in \fieldF \setminus {0}$ is a nonzero constant field element.} matrix polynomial \emph{cofactors} $\E(z)$ and $\F(z)$ which satisfy $\E(z)\L(z)\F(z) = \diag(\P(z), \I_n, \ldots, \I_n)$.  We write $\I_n$ for the $n \times n$ identity matrix here.

Linearizations preserve information not only about eigenvalues, but also eigenvectors and invariant factors.  Going further, a \emph{strong} linearization is one for which the matrix pencil $-z\L(1/z) = z\A-\B $ is a linearization for the reversal of $\P(z)$, namely $z^\ell \P(1/z)$.  Strong linearizations also preserve information about eigenstructure at infinity.

The Hermite normal form of a matrix polynomial $\L$ is an upper triangular matrix polynomial $\H$ unimodularly row equivalent to $\L$, which is to say $\L = \E\H$, where $\E$ is a matrix polynomial with $\det(\E)$ a nonzero constant. See ~\cite{storjohann1994computation} for properties and Hermite form algorithm description. 
We will chiefly use the computation of the Hermite normal form as a step in the process of discovering unimodular matrix polynomials $\E(z)$
and $\F(z)$ that demonstrate that $\L(z)$ linearizes $\P(z)$.

Another technique, which gives a stronger result (when you can do it) is to prove \emph{strict equivalence}.  We say matrix pencils $\L_1(z)$ and $\L_2(z)$ are \emph{strictly equivalent} if there exist constant unimodular matrices $\U$ and $\W$ with 
$\U \L_1(z) \W = \L_2(z)$.  We will cite and show some strict equivalence results in this paper, and prove a new strict equivalence for some Lagrange bases linearizations.

The paper~\cite{Dopico2020} introduces a new notion in their Definition~2, that of a \emph{local} linearization of \emph{rational} matrix functions: this definition allows $\E(z)$ and $\F(z)$ to be \emph{rational} unimodular transformations, and defines a local linearization only on a subset $\Sigma$ of $\fieldF$.  Specifically, by this definition, two rational matrices $\G_1(z)$ and $\G_2(z)$ are locally equivalent if there exist rational matrices $\E(z)$ and $\F(z)$, invertible for all $z \in \Sigma$, such that $\G_1(z) = \E(z)\G_2(z)\F(z)$ for all $z \in \Sigma$.

This allows exclusion of poles of $\E(z)$ or $\F(z)$, for instance.  It turns out that several authors, including~\cite{amiraslani2008linearization}, had in effect been using rational matrices and local linearizations for matrix polynomials.  For most purposes, in skilled hands this notion provides all the analytical tools that one needs.  However, as we will see, unimodular polynomial equivalence is superior in some ways.  This paper therefore is motivated to see how far such local linearizations, in particular for the Bernstein basis and the Lagrange interpolational bases, can be strengthened to unimodular matrix polynomials. The point of this paper, as a symbolic computation paper, is to see how experiments with built-in code for Hermite normal form can help us to discover such cofactors.

We write
\begin{equation}
B_{k}^{\ell}(z) = \binom{\ell}{k}z^k(1-z)^{\ell-k}\qquad 0 \le k \le \ell \label{eq:BernsteinDefinition}
\end{equation}
for the Bernstein polynomials of degree $\ell$.  This set of $\ell+1$ polynomials forms a basis for polynomials of grade~$\ell$.

We will use the convention of writing $\A_k$ for the coefficients when the matrix polynomial is expressed in the monomial basis or in another basis with a three-term recurrence relation among its elements, $\Y_k$ for the coefficients when the matrix polynomial is expressed in the Bernstein basis, and $\P_k$ for the coefficients when the matrix polynomial is expressed in a Lagrange basis.  We will use lower-case letters for scalar polynomials. We may write the (matrix) coefficient of $\phi_k(z)$ for an arbitrary basis element as $[\phi_k(z)]\P(z)$; for instance, the ``leading coefficient'' of $\P(z)$, considered as a grade $\ell$ polynomial, is $[z^\ell]\P(z)$.

\section{Explicit Cofactors}
We find cofactors for orthogonal bases, the Bernstein basis, and Lagrange interpolational bases, all modulo certain exceptions.  We believe all of these are new. These results are restricted to certain cases; say for the Bernstein basis when the coefficient $\Y_\ell = [B^\ell_\ell(z)]\P(z)$ (which is not the same as $[z^\ell]\P(z)$) is nonsingular. The method of this paper does not succeed in that case to find cofactors that demonstrate universal linearization.  In fact, however, it is known that this Bernstein companion pencil is indeed a true linearization; we will give two references with two different proofs, and give our own proof in section~\ref{sec:strictLagrange}.

Similarly, this method produces new cofactors for Lagrange interpolational bases, but in this case restricted to when the \emph{coefficients} $\P_k$ (which are actually the values of the matrix polynomial at the interpolational nodes) are all nonsingular.  We also have an \emph{algebraic} proof of equivalence, which we have cut for space reasons; this gives another proof that the Lagrange interpolational basis pencils used here are, in fact, linearizations. 
In section~\ref{sec:strictLagrange} we prove strict equivalence.
\subsection{Monomial Basis}
Given a potential linearization $\L(z) = z\C_1-\C_0$, it is possible to discover the matrices $\E(z)$ and $\F(z)$ by computing the generic \emph{Hermite Form}\footnote{The Smith form, which is related, is also useful here; but we found that the Maple implementation of the Hermite Form~\cite{storjohann1994computation} gave a simpler answer, although we had to compute the matrix $\F(z)$ separately ourselves because the Hermite form is upper triangular, not diagonal.}, with respect to the variable $z$, for instance by using a modestly sized example of~$\L(z)$ and a symbolic computation system such as Maple.  The generic form that we obtain is of course not correct on specialization of the symbolic coefficients, and in particular is incorrect if the leading coefficient is zero; but we may adjust this by hand to find the following construction, which we will not detail in this section but will in the next.

By using the five-by-five scalar case, with symbolic coefficients which we write as $a_k$ instead of $\A_k$ because they are scalar, we find that if
\begin{equation}
\E(z) =    \left[ \begin {array}{ccccc} 1&h_4&h_3&h_2&h_1\\ \noalign{\medskip}0
&0&0&0&-1\\ \noalign{\medskip}0&0&0&-1&-z\\ \noalign{\medskip}0&0&-1&-
z&-{z}^{2}\\ \noalign{\medskip}0&-1&-z&-{z}^{2}&-{z}^{3}\end {array}
 \right] 
\end{equation}
where $h_5 = a_5$ and $h_{k} = a_k + zh_{k+1}$ for $k=4$, $3$, $2$, $1$ are the partial Horner evaluations of $p(z) = a_5z^5 + \cdots + a_0$, and if
\begin{equation}
    \F(z) =  \left[ \begin {array}{ccccc} {z}^{4}&0&0&0&1\\ \noalign{\medskip}{z}^
{3}&0&0&1&0\\ \noalign{\medskip}{z}^{2}&0&1&0&0\\ \noalign{\medskip}z&
1&0&0&0\\ \noalign{\medskip}1&0&0&0&0\end {array} \right] \>,
\end{equation}
and if
\begin{equation}
    \L(z) =  \left[ \begin {array}{ccccc} za_{{5}}+a_{{4}}&a_{{3}}&a_{{2}}&a_{{1}}
&a_{{0}}\\ \noalign{\medskip}-1&z&0&0&0\\ \noalign{\medskip}0&-1&z&0&0
\\ \noalign{\medskip}0&0&-1&z&0\\ \noalign{\medskip}0&0&0&-1&z
\end {array} \right] \>,
\end{equation}
then we have $\E(z)\L(z)\F(z) = \diag(p(z), 1, 1, 1, 1)$.  Moreover, $\det \E(z) = \pm 1$ and $\det \F(z) = \pm 1$, depending only on dimension.

Once this form is known, it is easily verified for general grades and quickly generalizes to matrix polynomials, establishing (as is well-known) that this form (known as the \emph{second companion form}) is a linearization.

Note that the polynomial coefficients $a_k$ appear linearly in $\E(z)$ and the unimodular matrix polynomials $\E$ and $\F$ are thus universally valid.
\subsection{Three-term Recurrence Bases}
\label{recur-subsec}
The monomial basis, the shifted monomial basis, the Taylor basis, the Newton interpolational bases, and many common orthogonal polynomial bases all have three-term recurrence relations that, for $k\ge 1$, can be written
\begin{equation}
	z\phi_{k}(z) = \alpha_{k}\phi_{k+1}(z) + \beta_{k}\phi_{k}(z) + \gamma_{k}\phi_{k-1}(z) \>.
\end{equation}
All such polynomial bases require $\alpha_{k}\ne 0$.
For instance, the Chebyshev polynomial recurrence is usually written $T_{n+1}(z) = 2zT_n(z) - T_{n-1}(z)$ but is easily rewritten in the above form by isolating $zT_n(z)$, and all Chebyshev $\alpha_k = 1/2$ for $k>1$.
We 
refer the reader to section 18.9 of the Digital Library of Mathematical Functions (\url{dlmf.nist.gov}) for more. See also \cite{gautschi2016orthogonal}.

For all such bases, we have the linearization\footnote{For exposition, we follow Peter Lancaster's dictum, namely that the $5\times 5$ case almost always gives the idea.} $\L(z) = zC_1 - C_0$ where
\begin{equation}
    C_1 =
    \left[\begin{array}{c|cccc}
    	\dfrac{a_{5}}{\alpha_{4}} & & & & \\
    \hline
    & \I_4\\
    \end{array}\right]
\end{equation}
and
\begin{equation}
   C_0 = \left[\begin{array}{c|cccc}
    	-a_{4} + \dfrac{\beta_{4}}{\alpha_{4}}a_{5} & -a_{3} + \dfrac{\gamma_{4}}{\alpha_{4}}a_{5} & -a_{2} & -a_{1} & -a_{0} \\
    	\hline
        \alpha_{3} & \beta_{3} & \gamma_{3} & & \\
        & \alpha_{2} & \beta_{2} & \gamma_{2} & \\
        & & \alpha_{1} & \beta_{1} & \gamma_{1} \\
        & & & \alpha_{0} & \beta_{0}
    \end{array}\right] \>,
\end{equation}
In Maple we compute the Hermite normal form of a grade $5$ example with symbolic coefficients by the following commands:
\begin{lstlisting}
with(LinearAlgebra):
m := 5:
poly := add(a[k]*ChebyshevT(k, z), k = 0 .. m):
(C0, C1) := CompanionMatrix(poly, z):
\end{lstlisting}
That procedure does not use the same convention we use here and so we apply the standard involutory permutation (SIP) matrix to it and transpose it to place the polynomial coefficients in the top row.
\begin{lstlisting}
J := Matrix( m, m, (i,j)->`if`( i+j=m+1, 1, 0 )):
R := C1*z - C0:
JRJT := Transpose((J . R) . J):
\end{lstlisting}
Now compute the generic Hermite normal form:
\begin{lstlisting}
(HH, UU) := HermiteForm( JRJT,z,output=['H','U']):
\end{lstlisting}
That returns matrices such that $\U\L = \H$, or $\L = \U^{-1}\H$.  We now look at their structure.
\begin{lstlisting}
mask := proc(A::Matrix) 
  map(t -> `if`(t = 0, 0, x), A);
end proc:
mask( HH );
\end{lstlisting}
This produces
\begin{equation}
    \left[\begin{array}{ccccc}x  & 0 & 0 & 0 & x  \\0 & x  & 0 & 0 & x  \\0 & 0 & x  & 0 & x  \\0 & 0 & 0 & x  & x  \\0 & 0 & 0 & 0 & x  \end{array}\right]
\end{equation}
and a separate investigation shows the diagonal entries are (as is correct in the generic case) all just $1$, until the lower corner entry which is a monic version of the original polynomial. The matrix $\U$ has a simple enough shape in this case, but $\U^{-1}$ is more convenient:
\begin{lstlisting}
mask( UU^(-1) );
\end{lstlisting}
produces
\begin{equation}
\left[\begin{array}{ccccc}x  & x  & x  & x  & x  \\x  & x  & x  & 0 & 0 \\0 & x  & x  & x  & 0 \\0 & 0 & x  & x  & 0 \\0 & 0 & 0 & x  & 0 \end{array}\right]\>.
\end{equation}
Because the first $m-1$ columns of $\H$ is a subset of the identity matrix, the first $m-1$ columns of $\U^{-1}$ are the same as the first $m-1$ columns of $\L$.  Since the final column of $\U^{-1}$ has only one nonzero entry, at the top, call that $u_0$.  Call the entries in the final column of $\H$ in descending order $h_{m-1}$, $h_{m-2}$,$\ldots$, $h_1$; let us suppose that the final entry is a multiple $cp(z)$ of the scalar polynomial.

Multiplying out $\U^{-1}\H$ and equating its final column to the final column of $\L(z)$ gives us a set of (triangular, as it happens) equations in the unknowns.  These allow us to deduce a general form for $\U^{-1}(z)$, and to prove it is unimodular.  
Once we have $\U$ and $\H$ with $\U\L(z) = \H$, construction of unimodular $\E(z)$ and $\F(z)$ so that $\E(z)\L(z)\F(z) = \diag(\P(z), \I_n, \ldots,\I_n)$ is straightforward.

The results are slightly simpler in this case to present as inverses: Here we show the Chebyshev scalar case.
\begin{equation}
    \E^{-1} = \left[\begin{array}{ccccc}1 & a_{1} & a_{2} & a_{3}-a_{5} & 2 z a_{5}+a_{4} \\0 & 0 & -\frac{1}{2} & z  & -\frac{1}{2} \\0 & -\frac{1}{2} & z  & -\frac{1}{2} & 0 \\0 & z  & -\frac{1}{2} & 0 & 0 \\0 & -1 & 0 & 0 & 0 \end{array}\right]
\end{equation}
and
\begin{equation}
    \F^{-1} = \left[\begin{array}{ccccc}0 & 0 & 0 & 0 & 1 \\0 & 0 & 0 & 1 & -T_{1}\! \left(z \right) \\0 & 0 & 1 & 0 & -T_{2}\! \left(z \right) \\0 & 1 & 0 & 0 & -T_{3}\! \left(z \right) \\1 & 0 & 0 & 0 & -T_{4}\! \left(z \right) \end{array}\right].
\end{equation}
It is not immediately obvious that $\E(z)$ is unimodular, but it is. For $\F(z)$ it is obvious.

Notice again that the result is linear in the unknown polynomial coefficients, and that we therefore have a universal equivalence (once generalized to the matrix polynomial case, and of arbitrary dimension, which is straightforward).
\section{Bernstein Basis}
\label{bernstein-subsec}
The set of Bernstein polynomials $B^\ell_k(z)$ in equation~\eqref{eq:BernsteinDefinition} is a set of $\ell+1$ polynomials each of exact degree $\ell$ that together forms a basis for polynomials of grade~$\ell$ over fields $\fieldF$ of characteristic zero. 
Bernstein polynomials have many applications, for example in Computer Aided Geometric Design (CAGD), and many important properties including that of optimal condition number over all bases positive on $\left[0, 1\right]$. They do not satisfy a simple three term recurrence relation of the form discussed in Section~\ref{recur-subsec}, although they satisfy an interesting and useful ``degree-elevation'' recurrence, namely
\begin{equation}
     (j+1)B_{j+1}^n(z) + (n-j)B_{j}^n(z) = n B_{j}^{n-1}(z)\>,
\end{equation}
which specifically demonstrates that a sum of Bernstein polynomials of degree~$n$ might actually have degree strictly less than~$n$.
See~\cite{farouki2012bernstein},~\cite{farouki1996optimal}, and~\cite{farouki1987numerical} for more details of Bernstein bases.

A Bernstein linearization for $p_{5}(z) = \sum_{k=0}^{5}\scof{k}B_{k}^{5}(z)$ is 
$\L(z) =$
\begin{align}\label{eq:BernsteinPencil}
   \begin{bmatrix}
    	 \frac{1}{5}\scof{5}z
    	 +\scof{4}(1-z) &
    	 \scof{3}(1-z) &
    	 \scof{2}(1-z) &
    	 \scof{1}(1-z) &
    	 \scof{0}(1-z) \\
        z-1 & \frac{2}{4}z & & & \\
        & z-1 & \frac{3}{3}z & & \\
        & & z-1 & \frac{4}{2}z & \\
        & & & z-1 & \frac{5}{1}z
    \end{bmatrix} \>.
\end{align}
    

The pattern on the diagonal is written in unreduced fractions so that the general pattern is more clear.

For a construction of rational unimodular $\E(z)$ and $\F(z)$ that show that this is a \emph{local} linearization, in the case that no eigenvalue occurred at the end of the Bernstein interval ($z=1$), see~\cite{amiraslani2008linearization}. The discussion in~\cite{mackey2016linearizations} shows by strict equivalence that it is in fact a global strong linearization independently of any singularities, or indeed independently of regularity of the matrix polynomial.
We will give here only an explicit construction of unimodular polynomial matrices $\E(z)$ and $\F(z)$, again unless $z=1$, and requiring that the ``leading'' block $[B^\ell_\ell(z)](\P(z)) = \Y_\ell$ be nonsingular.

The linearization used here was first analyzed in~\cite{jonsson2001eigenvalue} and~\cite{jonsson2004solving} (as a companion pencil for scalar polynomials) and has been further studied and generalized in~\cite{mackey2016linearizations}. One of the present authors independently invented and implemented a version of this linearization in the Maple command \texttt{CompanionMatrix} (except using $\P^{\mathrm{T}}(z)$, and flipped from the above form) in about $2004$. For a review of Bernstein linearizations, see the aforementioned~\cite{mackey2016linearizations}.
For a proof of their numerical stability, see the original thesis~\cite{jonsson2001eigenvalue}.  

\subsection{Hermite form of the linearization}
We use the same approach as before, probing with a small example and computing its Hermite normal form explicitly to suggest the correct form.

The resulting suggested form for $\U^{-1}$ is
\begin{equation}
    \left[\begin{array}{ccccc}x  & x  & x  & x  & x  \\x  & x  & 0 & 0 & x  \\0 & x  & x  & 0 & x  \\0 & 0 & x  & x  & x  \\0 & 0 & 0 & x  & x  \end{array}\right]
\end{equation}
which is more complicated than before, because the final column is full (as before, the first $\ell-1$ columns merely copy the companion pencil $\L(z)$). Moreover, this time the entries in $\U^{-1}$ are \emph{polynomial} functions of the $y_k$, $0 \le k < \ell$, not just linear; but involve negative powers of $\scof{\ell}$ (this equivalence appears to be new: the treatment in~\cite{mackey2016linearizations} is implicit, while the treatment in~\cite{amiraslani2008linearization} only gives a local linearization).  This means that in the case $\Y_\ell$ is singular, $z=1$ is an eigenvalue and something else must be done to linearize the matrix polynomial.

We find a recurrence relation for the unknown blocks in both the purported Hermite normal form and in the inverse of the cofactor. Put $    \U^{-1} = $
\begin{equation}
    \begin{bmatrix}
    z\Y_\ell/\ell-(z-1)\Y_{\ell-1} & -(z-1)\Y_{\ell-2} & \cdots & -(z-1)\Y_{1} & \W_\ell \\
    -(z-1)\I_n & 2z\I_n/(\ell-1) & & & \W_{\ell-1} \\
     & -(z-i)\I_n &\ddots & & \vdots \\
     & & & -(z-1)\I_n & \W_0 \\
    \end{bmatrix}
\end{equation}
which, apart from the final column, is the same as the linearization $\L(z)$, and the Hermite normal form analogue as
\begin{equation}
    \H = \begin{bmatrix}
    \I_n & & & & \H_{\ell-1} \\
         & \I_n & & & \H_{\ell-2} \\
         &      & \ddots & & \vdots \\
         & & & \I_n & \H_1 \\
         & & &      & \P(z)
    \end{bmatrix}\>.
\end{equation}
Multiplying out $\U^{-1}\H$ we find that for the lower right corner block
\begin{equation}
    -(z-1)\H_1 + \W_0 \P(z) = \ell z \I_n\>. \label{eq:corner}
\end{equation}
A separate investigation shows that $\W_0$ must be a constant matrix.  Evaluating that equation at $z=1$ shows that $\W_0 = \ell \P^{-1}(1)$, \emph{so $\P(1)$ must be nonsingular} for this form to
hold.  Once $\W_0$ is identified, equation~\eqref{eq:corner} can be solved uniquely for the grade $\ell-1$ matrix polynomial $\H_1(z)$:
\begin{equation}
    \H_1(z) = \frac{\ell}{z-1}\left( \P^{-1}(1)\P(z) - z\I_n\right)\>.
\end{equation}
Division is exact there, as can be checked by expanding the numerator of the right-hand side in Taylor series about $z=1$.

The other block entries in the multiplied-out equation give a triangular set of recurrences for the $\W_k$ and the $\H_k(z)$ that are solvable in the same way and under the same condition, namely that $\P(1)$ must be nonsingular. 

Once we have $\U$ and $\H$ with $\U\L(z) = \H$, construction of unimodular $\E(z)$ and $\F(z)$ so that $\E(z)\L(z)\F(z) = \diag(\P(z), \I_n, \ldots,\I_n)$ is straightforward.
\subsection{Strict Equivalence\label{sec:strictBernstein}}
For the pencil in equation~\eqref{eq:BernsteinPencil} we exhibit the strict equivalence by the unimodular matrices $\U$ and $\W$ giving $\L_m(z)=\U \L_{B}(z)\W$, below; for ease of understanding, we actually show $\U^{-1}$ which shows its relationship to $\W$ more clearly:
\begin{equation}
    \W = \left[\begin{array}{ccccc}5 & 0 & 0 & 0 & 0 \\-10 & 10 & 0 & 0 & 0 \\10 & -20 & 10 & 0 & 0 \\-5 & 15 & -15 & 5 & 0 \\1 & -4 & 6 & -4 & 1 \end{array}\right]
\end{equation}
and
\begin{equation}
\U^{-1} = \left[\begin{array}{ccccc}1 & h_3 & h_2 & h_1 & -y_{0} \\0 & 5 & 0 & 0 & 0 \\0 & -10 & 10 & 0 & 0 \\0 & 10 & -20 & 10 & 0 \\0 & -5 & 15 & -15 & 5 \end{array}\right]\>,
\end{equation}
where $h_3 = -10 y_{3}+20 y_{2}-15 y_{1}+4 y_{0}$, $h_2 = -10 y_{2}+15 y_{1}-6 y_{0}$, and $h_1 = -5 y_{1}+4 y_{0}$.
Generalizing these to matrix polynomials is straightforward, as is generalizing these to arbitrary grade.

The paper~\cite{mackey2016linearizations} contains, in its equation~(3.6), a five by five example including tensor products that shows how to find a strict equivalence of the pencil here to the strong linearizations they construct in their paper.  Both the results of that paper and the construction given by example above demonstrate that the Bernstein linearization here is a strong linearization, independently of the singularity of $\P(1)$ and indeed independently of the regularity of the matrix polynomial.

\subsection{A new reversal\label{sec:bernreversal}}
We here present a slightly different \emph{reversal}, namely rev~$p(z) = (z+1)^\ell p(1/(z+1))$ of a polynomial of grade~$\ell$ expressed in a Bernstein basis, instead of the standard reversal $z^\ell p(1/z) $.
This new reversal has a slight numerical advantage if all the coefficients of $p(z)$ are the same sign.  We also give a proof that the linearization of this reversal is the corresponding reversal of the linearization, thus giving a new independent proof that the linearization is a strong one.
This provides the details of the entries in the 
unimodular matrix polynomials $\E(z)$ and $\F(z)$ with $\E(z)\L(z)\F(z) = \diag{\P(z),\I_n,\ldots,\I_n}$ constructed above.

A short computation shows that if 
\begin{equation}
    p(z) = \sum_{k=0}^\ell y_k B_k^\ell(z)
\end{equation}
then
\begin{equation}
    \mathrm{rev}\, p(z) = (z+1)^\ell p\left(\frac{1}{z+1}\right) = \sum_{k=0}^\ell d_k B_k^\ell(z)
\end{equation}
where
\begin{equation}
    d_k = \sum_{j=0}^k \binom{ k}{j } y_{\ell-j}\>,
\end{equation}
whereas the coefficients of the standard reversal are, in contrast,
\begin{equation}
    e_k = \sum_{m=0}^{\ell-k} (-1)^m \binom{\ell-k}{m} y_{\ell-m-k}
\end{equation}
which has introduced sign changes, which may fail to preserve numerical stability if all the $y_k$ are of one sign.  A further observation is that the coefficient $d_0$ only involves $y_\ell$, while $e_0$ involves all $y_k$; $d_1$ involves $y_\ell$ and $y_{\ell-1}$ while $e_1$ involves all but $y_0$, and so on; in
that sense, this new reversal has a more analogous behaviour to
the monomial basis reversal, which simply reverses the list of coefficients.   

For interest, we note that if $(\A,\B)$ is a linearization for $p(z)$ so that $p(z) = \det( z\B - \A)$, then reversing the linearization by this transformation is not a matter of simply interchanging $\B$ and $\A$: 
\begin{align}
(z+1)^\ell p\left(\frac{1}{z+1}\right) &= (z+1)^\ell\det \left(\frac{1}{z+1} \B - \A \right) \nonumber\\
&= \det( \B - (z+1)\A ) \nonumber\\
&= \det( \B-\A - z \A ) 
\end{align}
and so the corresponding reversed linearization is $(\A, \B-\A)$.  The
sign change is of no importance.

Suppose that the Bernstein linearization of $p(z)$ is $(\A,\B)$ and that the Bernstein linearization of rev~$p(z)$ is $(\A_R,\B_R)$.  That is, the matrices $\A_R$ and $\B_R$ have the same form as that of $\A$ and $\B$, but where $(\A,\B)$ contain $y_k$s the matrices $(\A_R,\B_R)$ contain $d_k$s. 
To give a new proof that the Bernstein linearization is actually a strong linearization, then, we must find a pair of unimodular matrices $(\U,\W)$ which have $\U \A_R \W = \B-\A$ and $\U\B_R\W = \A$, valid for all choices of coefficients $y_k$ (which determine the corresponding reversed coefficients $d_k$ by the formula above).  


First, it simplifies matters to deal not with $\W$ but rather with $\W^{-1}$. Then, our defining conditions become
\begin{align}
    \U \A_R &= \left(\B-\A\right)\W^{-1} \\
    \U \B_R &= \A\W^{-1}\>,
\end{align}
which are linear in the unknowns (the entries of $\U$ and of $\W^{-1}$).  By inspection of the first few dimensions, we find
that $\U$ and $\W^{-1}$ have the following form (using the six-by-six case, for variation, to demonstrate).  The anti-diagonal of the general $\U$ has entries $-(\ell-i+1)/i$ for $i=1$, $2$, $\ldots$, $\ell$.
\begin{equation}
\U =     \left[ \begin {array}{cccccc} 0&0&0&0&0&-6\\ \noalign{\medskip}0&0&0&0
&-{\frac{5}{2}}&u_{{2,6}}\\ \noalign{\medskip}0&0&0&-{\frac{4}{3}}&u_{
{3,5}}&u_{{3,6}}\\ \noalign{\medskip}0&0&-{\frac{3}{4}}&u_{{4,4}}&u_{{
4,5}}&u_{{4,6}}\\ \noalign{\medskip}0&-{\frac{2}{5}}&u_{{5,3}}&u_{{5,4
}}&u_{{5,5}}&u_{{5,6}}\\ \noalign{\medskip}-{\frac{1}{6}}&u_{{6,2}}&u_
{{6,3}}&u_{{6,4}}&u_{{6,5}}&u_{{6,6}}\end {array} \right] 
\end{equation}
and
\begin{equation}
    \W^{-1} =  \left[ \begin {array}{cccccc} 0&0&0&0&z_{{1,5}}&z_{{1,6}}
\\ \noalign{\medskip}0&0&0&z_{{2,4}}&z_{{2,5}}&z_{{2
,6}}\\ \noalign{\medskip}0&0&z_{{3,3}}&z_{{3,4}}&z_{
{3,5}}&z_{{3,6}}\\ \noalign{\medskip}0&z_{{4,2}}&z_{
{4,3}}&z_{{4,4}}&z_{{4,5}}&z_{{4,6}}
\\ \noalign{\medskip}z_{{5,1}}&z_{{5,2}}&z_{{5,3}}&
z_{{5,4}}&z_{{5,5}}&z_{{5,6}}\\ \noalign{\medskip}0&0
&0&0&0&1\end {array} \right] 
\end{equation}
One can then guess the explicit general formulae
\begin{align}
    u_{i,j} &= -\frac{(\ell-i+1)}{i}\binom{i}{\ell+1-j} \qquad 1\le i,j\le \ell\\
    z_{i,j} &= -\frac{(\ell-i)}{j}\binom{i}{\ell-j} \qquad 1\le i,j \le \ell-1\\
    z_{i,\ell} &= d_i - \frac{(\ell-i)}{\ell}y_\ell\>,\qquad 1\le i \le \ell-1
\end{align}
and then prove that these are not only \emph{necessary} for the equations above, but also \emph{sufficient}.  The matrix $\B-\A$ is diagonal and gives a direct relationship between the triangular block in $\W^{-1}$ and a corresponding portion of $\U$; the other equation gives a recurrence relation for the entries of $\U$.  Comparison of the final columns of the products gives an explicit formula for the final column of $\W^{-1}$ and an explicit formula for the entries of $\U$ by comparison of the coefficients of the symbols $y_k$; this formula can be seen to verify the recurrence relation found earlier, closing the circle and establishing sufficiency.  Both matrices $\U$ and $\W$ have determinant $\pm 1$: $\lfloor \ell/2 \rfloor$ row-permutations brings $\U$ to upper triangular form and the determinant $(-1)^{\ell}$ (times $(-1)^{\lfloor \ell/2\rfloor}$) can be read off as the product of the formerly anti-diagonal elements, and similarly for the $[1:\ell-1,1:\ell-1]$ block of $\W^{-1}$ which gives a sign $(-1)^{\ell-1+\lfloor (\ell-1)/2 \rfloor}$.

\section{Lagrange Interpolational Bases}
A useful arrowhead companion matrix pencil for polynomials given in a Lagrange basis was given in~\cite{corless2004generalized,corless2004bernstein}.  Later, Piers Lawrence recognized that the mathematically equivalent pencil re-ordered by similarity transformation by the standard involutory permutation (SIP) matrix so that the arrow was pointing up and to the left was numerically superior, in that one of the spurious infinite eigenvalues will be immediately deflated---without rounding errors---by the standard QZ algorithm~\cite{lawrence2012numerical}.  We shall use that variant in this paper.  

For expository purposes, consider interpolating a scalar polynomial $p(z)$ on the four distinct nodes $\tau_k$ for $0 \le k \le 3$.  In the first barycentric form~\cite{trefethen} this is
\begin{equation}\label{lagrange}
p(z) = w(z) \sum_{k=0}^3 \frac{\beta_k p_k}{z-\tau_k} 
= \sum_{k=0}^3 p_k w_k(z),
\end{equation}
where the node polynomial $w(z) = (z-\tau_0)(z-\tau_1)(z-\tau_2)(z-\tau_3)$, and $w_k(z) = \beta_k w(z)/(z-\tau_k)$, and the barycentric weights $\beta_k$ are
\begin{equation}
    \beta_k = \prod_{j=0 \& j\neq k}^3 \frac{1}{\tau_k-\tau_j}\>.
\end{equation}
Then a Schur complement with respect to the bottom right $4\times 4$ block shows that if
\begin{equation}
    L(z) = \left[\begin{array}{ccccc}0 & -p_{3} & -p_{2} & -p_{1} & -p_{0} \\\beta_3 & z -\tau_{3} & 0 & 0 & 0 \\\beta_2 & 0 & z -\tau_{2} & 0 & 0 \\\beta_1 & 0 & 0 & z -\tau_{1} & 0 \\\beta_0 & 0 & 0 & 0 & z -\tau_{0} \end{array}\right]
\end{equation}
then $\det L(z) = p(z)$.  By exhibiting a \emph{rational} unimodular equivalence, the paper~\cite{amiraslani2008linearization} showed that the general form of this was at least a local linearization for matrix polynomials $\P(z)$, and also demonstrated that this was true for the reversal as well, showing that it was a strong (local) linearization.  They also gave indirect arguments, equivalent to the notion of patched local linearizations introduced in~\cite{Dopico2020}, showing that the construction gave a genuine strong linearization.  Here, we wish to see if we can explicitly construct unimodular matrix polynomials $\E(z)$ and $\F(z)$ which show equivalence, directly demonstrating that this is a linearization.
\subsection{Hermite form of the linearization}
As we did for the monomial basis, we start with a scalar version. 
We compute the Hermite normal form $\H$, and the transformation matrix $\U$ so that $\U \L(z) = \H$, with Maple to give clues to find the general form.  When we do this for the grade $3$ example above we find that the form of $\U$ is not helpful, but that the form of $\U^{-1}$ and $\H$ are:
\begin{equation}
    L(z) = \left[\begin{array}{ccccc}0 & x  & x  & x  & 0 \\x  & x  & 0 & 0 & x  \\x  & 0 & x  & 0 & x  \\x  & 0 & 0 & x  & x  \\x  & 0 & 0 & 0 & 0 \end{array}\right]
    \left[\begin{array}{ccccc}1  & 0 & 0 & 0 & x  \\0 & 1  & 0 & 0 & x  \\0 & 0 & 1  & 0 & x  \\0 & 0 & 0 & 1  & x  \\0 & 0 & 0 & 0 & x  \end{array}\right]\>.
\end{equation}
As is usual, the generic Hermite normal form contains $p(z)$ in the lower right corner; on specialization of the polynomial coefficients this can change, of course.  We return to this point later.

This gives us enough information to conjecture the general form in the theorem below, and a proof follows quickly.

\begin{theorem}
If
\begin{equation}
    L(z) = \left[\begin{array}{ccccc}0 & -\P_{\ell} & -\P_{\ell-1} & \ldots & -\P_{0} \\\beta_{\ell}\I_n & (z -\tau_{\ell})\I_n & 0 & 0 & 0 \\\beta_{\ell-1}\I_n & 0 & (z -\tau_{\ell-1})\I_n & 0 & 0 \\\vdots &  &  & \ddots &  \\\beta_0\I_n & 0 & 0 & 0 & (z -\tau_{0})\I_n \end{array}\right]
\label{eq:LagrangePencil}
\end{equation}
and no $\P_k$ is singular then $L(z) = \U^{-1} \H$ where $\U^{-1} =$
\begin{equation*}
 \left[\begin{array}{cccccc}
    0 & -\P_{\ell} & -\P_{\ell-1} & \ldots & -\P_{1} & 0  \\
    \beta_{\ell}\I_n & (z -\tau_{\ell})\I_n & 0 & 0 & \cdots & \U_\ell \\
    \beta_{\ell-1}\I_n & 0 & (z -\tau_{\ell-1})\I_n & 0 & \cdots & \U_{\ell-1} \\
    \vdots &  &  & \ddots & & \\
    \beta_1\I_n & 0 & 0 & \cdots & (z-\tau_1)\I_n & \U_1 \\
    \beta_0\I_n & 0 & 0 & 0 & \cdots & 0  \end{array}\right]\>,
\end{equation*}
and $\H$ is the identity matrix with its final column replaced with a block matrix that can be partitioned as
\begin{displaymath}
\H =     \left[\begin{array}{cccccc}\I_n  &  &  &  & & \G  \\ & \I_n  &  &  &  & \H_\ell  \\ &  & \I_n  &  &  & \H_{\ell-1}  \\ &  &  & \ddots &  & \vdots  \\
 &   &   &  & \I_n & \H_1 \\
 &   &   &  &    &  \P(z)  \end{array}\right]\>.
\end{displaymath}
Moreover,
\begin{equation}
    \G = - \frac{1}{\beta_0}(z-\tau_0)\I_n\>,
\end{equation}
\begin{equation}
    \U_k = -\frac{\beta_k}{\beta_0}(\tau_k-\tau_0)\P_k^{-1}\>, \label{eq:Uk}
\end{equation}
for $1 \le k \le \ell$, and
\begin{equation}
    \H_k = \frac{1}{z-\tau_k}\left( \beta_k \G - \U_k\P(z)\right)\>.\label{eq:Hk}
\end{equation}
Note that the definition of $\U_k$ in equation~\eqref{eq:Uk} ensures that the division in equation~\eqref{eq:Hk} is exact, and therefore $\H_k$ is a matrix polynomial of grade $\ell-1$.
Furthermore, $\U$ is unimodular.
\end{theorem}
\begin{proof}
Block multiplication of the forms of $\U^{-1}$ and $\H$ gives $\ell+1$ block equations:
\begin{align}
    -\sum_{k=1}^\ell \P_k \H_k &= \P_0 \label{eq:P0eq}\\
    \beta_k \G + (z-\tau_k)\H_k + \U_k \P(z) &= 0, 1 \le k \le \ell \\
   \beta_0\G &= (z-\tau_0)\I_n \>.
\end{align}
The last block equation identifies $\G$.
Putting $z=\tau_k$ in each of the middle block equations identifies each $\U_k$.  Once that has been done, the middle block equations define each $\H_k$.  All that remains is to show that these purported solutions satisfy equation~\eqref{eq:P0eq} 
and that the resulting matrix $\U$ is unimodular. 

We will need the following facts about the node polynomial $w(z) = \prod_{k=0}^\ell (z-\tau_k)$, the barycentric weights $\beta_k$, and the first barycentric form for $\P(z)$:
\begin{equation}
    \frac{1}{w(z)} = \sum_{k=0}^\ell \frac{\beta_k}{z-\tau_k}\>, \label{eq:parfrac}
\end{equation}
\begin{equation}
    \frac{z}{w(z)} = \sum_{k=0}^\ell \frac{\beta_k\tau_k}{z-\tau_k}\>,
    \label{eq:zwz}
\end{equation}
and
\begin{equation}
    \P(z) = {w(z)}\sum_{k=0}^\ell \frac{\beta_k}{z-\tau_k}\P_k\>.
\end{equation}
By substituting $z=0$ in equation~\eqref{eq:zwz} we find that the sum of the barycentric weights is zero\footnote{This is true even if one of the $\tau_k$ is zero, inductively because one factor of $z$ then cancels in the numerator and denominator on the left hand side and the result is one degree lower.}.

We now substitute equations~\eqref{eq:Hk}--\eqref{eq:Uk} into the left hand side of equation~\eqref{eq:P0eq} to get
\begin{equation}
    \sum_{k=1}^\ell \frac{1}{z-\tau_k}\P_k\left(
    -\frac{\beta_k}{\beta_0}(z-\tau_0) + \frac{\beta_k}{\beta_0}(\tau_k-\tau_0) \P_k^{-1}\P(z)
    \right) \>.
\end{equation}
Expanding this, we have
\begin{align}
\mathrm{LHS} = & 
-\frac{(z-\tau_0)}{\beta_0}\sum_{k=1}^\ell \frac{\beta_k}{z-\tau_k}\P_k + 
\sum_{k=1}^\ell \frac{\beta_k(\tau_k-\tau_0)}{\beta_0(z-\tau_k)} \P(z)
\nonumber\\
= & -\frac{(z-\tau_0)}{\beta_0}\left[\frac{1}{w(z)}\P(z) - \frac{\beta_0}{z-\tau_0}\P_0\right]   \nonumber\\
 &\qquad + \frac{1}{\beta_0}\left[\frac{z}{w(z)} - \frac{\beta_0\tau_0}{z-\tau_0}\right]\P(z) - \frac{\tau_0}{\beta_0}\left[\frac{1}{w(z)}-\frac{\beta_0}{z-\tau_0}\right]\P(z)\nonumber\\
&= \P_0\>.
\end{align}
This shows that we have found a successful factoring analogous to the scalar Hermite normal form.

All that remains is to show that $\U(z)$ is unimodular.  We use the Schur determinantal formula, and identify the $\beta_k(\tau_k-\tau_0)$ as the barycentric weights on the nodes with $\tau_0$ removed and we see that up to sign (depending on dimension) the determinant
simplifies to $1$.  
This completes the proof.
\end{proof}
As a corollary, we can explicitly construct unimodular matrix polynomials $\E(z)$ and $\F(z)$ from these factors showing that, if each $\P_k$ is nonsingular, $L(z)$ is equivalent to $\diag(\P(z),\I_n,\ldots,\I_n)$. 

As in~\cite{amiraslani2008linearization} we may use $LU$ factoring to show that $\P$ is also equivalent to $\L$ at the nodes, essentially using Proposition 2.1 of~\cite{Dopico2020}.  We also have a purely algebraic proof based on local equivalence of Smith forms, not shown here.

\subsection{Linearization equivalence 
via local Smith form}
In this section we approach the equivalence of $\P$ to its linearization via localizations. 
The argument does not require that $\P$ or its evaluations $\P_k$ be nonsingular.
However, this method is less explicit about the form of the unimodular cofactors.  Also, a stronger result, strict equivalence, is shown in section \ref{sec:strictLagrange}. The line of argument here, that local information may be brought together to obtain a global equivalence, may be useful for other equivalence problems lacking strict equivalence.

The localizations here are in the algebraic sense. 
$\fieldF[z]$ is a principal ideal ring as is each localization, $\fieldF[z]_f = \fieldF[z]/M_f$, at an irreducible $f\in \fieldF[z]$, where $M_f$ denotes the multiplicatively closed set of all $g \in \fieldF[z]$ relatively prime to $f$. 
The ideals of $\fieldF[x]_f$ are the multiples of a given power of $f$, $I_k = f^k\fieldF[x]_f.$ 
Working over such a localization, unimodular cofactors may have  denominators $d(z)$ relatively prime to $f$.  
Then this algebraically local solution is a local solution in the analytic sense of section \ref{sect:defs}, being valid for all $z$ not a zero of $d$. 
\begin{lemma}Matrices $A,B \in \fieldF[z]^{n\times n}$ are equivalent if and only if they are equivalent over every localization $\fieldF[z]_f$ at an irreducible $f(z)$.
\end{lemma}
\begin{proof}
Let $\S = \diag(s_1, \ldots,s_n)$ be the Smith normal form of $A$.   
There are unimodular $\E, \F\in \fieldF[z]$ such that $\A = \E\S\F$. Let $f(z)$ be an irreducible and 
suppose $g_k[z]$ and $e_k$ are such that $s_k = f^{e_k}g_k$ with $g$ relatively prime to $f$. Then the matrix $\U = \diag(g_1, \ldots, g_n)$ is unimodular over $\fieldF[z]_f$ and 
$\A = \E\U\S_f\F$, where $\S_f=\diag(f^{e_1},\ldots,f^{e_n})$ is in Smith normal form and is the unique Smith form of $\A$ locally at $f$.  Thus the powers of $f$ in the Smith form over $\fieldF[z]$ form precisely the Smith form locally over $\fieldF[z]_f$.  Matrices $\A,\B$ are equivalent precisely when they have the same Smith form.  As we've just shown, this happens if and only if they have the same Smith forms locally at each irreducible $f$.
\end{proof}
We remark that it would also work to approximate the localizations and do computations over $\fieldF[x]/\langle f^e\rangle$ for sufficiently large $e$.  $e = n\ell$ would do since the degree of any minor of $P$ is bounded by $n\ell$.  For a thorough treatment on the local approach to Smith form, see \cite{WILKENING20111}.

Assuming that $\L$ is equivalent to 
$\diag(\P, \I_n,\ldots, \I_n)$ (which we show next), 
one can readily produce unimodular cofactors using a Smith form algorithm.  Let $\S$ be the Smith form of $\P$; then the Smith form of $\L$ is $\diag(\S, \I_n, \ldots, \I_n)$.  Let the unimodular cofactors be $\A$, $\B$, $\C$, $\D$ such that $\P=\A\S\B$ and $\L = \C\diag(\S,\I_n,\ldots, \I_n)\D$.  
Then, using the cofactors 
$$\E = \C\diag(\A^{-1}, \I_n, \ldots,\I_n) \mbox{ and } 
F = \diag(\B^{-1},\I_n,\ldots,\I_n)\D,$$ one has $L = \E\diag(\P,\I_n, \ldots,\I_n)\F$.

\begin{theorem}
Let $\tau$ be a list of $\ell+1$ distinct nodes in $\fieldF$.
Let $\P\in \fieldF[z]^{n\times n}$ of degree no larger than $\ell$. Then $\L$ is equivalent to $\diag(\P, \I_n, \ldots, \I_n)$, in which $\L$ is the Lagrange interpolation linearization of $\P$ and there are $\ell+1$ identity blocks in the block diagonal equivalent.
\end{theorem}
\begin{proof} We'll show the equivalence for each localization at an irreducible.  From this the equivalence over $\fieldF[z]$ follows. 
We show the equivalence locally at any $f$ relatively prime to $w_\ell$.  
This captures 
 the general case where $f$ is relatively prime to all $z - \tau_k$ as well as 
 the special case $f = z-\tau_\ell$.  
 The cases for $f = z -\tau_k, k < \ell$ follow by symmetry.
 
Let $\Z = \diag((z-\tau_{\ell-1})I_n, \ldots,(z-\tau_0)I_n),  
\R = \begin{bmatrix}
\P_{\ell-1}  & \ldots & \P_0
\end{bmatrix}$,
and 
$\B = \begin{bmatrix}
\beta_{\ell-1} \I_n & \ldots & \beta_0\I_n
\end{bmatrix}^T$. 
Then we have the block form
$$
\L =
\begin{bmatrix}
0 & \P_\ell & \R \\
-\beta_\ell\I_n & (z-\tau_\ell)\I_n & 0\\
-\B & 0 & \Z \\
\end{bmatrix}
$$
Note that $w$ is a unit locally at $f$ so that $Z$ is unimodular.
Thus $\L$ is equivalent to $\S\oplus Z$ and to $\S\oplus I_{n\ell}$, where $\S$ is the following Schur complement of $\Z$ in $\L$.
$$
\begin{bmatrix}
0 & \P_\ell\\
-\beta_\ell\I_n & (z-\tau_\ell)\I_n\\
\end{bmatrix}
-
\begin{bmatrix}
\R \\ 0 \\ 
\end{bmatrix}
\Z^{-1}
\begin{bmatrix}
-\B & 0 \\
\end{bmatrix}
=
\begin{bmatrix}
\R\Z^{-1}\B & \P_\ell\\
-\beta_\ell\I_n & (z-\tau_\ell)\I_n\\
\end{bmatrix}.
$$
 Then we may manipulate this block into the desired form $\diag(\P, \I_n)$.
$$
\begin{bmatrix}
-\I_n & -\beta_\ell^{-1}\R\Z^{-1}\B\\
0 & -\beta_\ell^{-1}\I_n\\
\end{bmatrix}
\begin{bmatrix}
\R\Z^{-1}\B & \P_\ell\\
-\beta_\ell\I_n & (z-\tau_\ell)\I_n\\
\end{bmatrix}
\begin{bmatrix}
-w\I_n & \I_n\\ 
-w_\ell\I_n & 0\\
\end{bmatrix}.
$$
$$
=
\begin{bmatrix}
w\R\Z^{-1}\B + w_\ell\P_\ell  & 0 \\
0 & \I_n
\end{bmatrix}.
$$
Expanding the leading block, we see that it is $\P$:
$$
w \sum_{k=0}^{\ell-1}\P_k (z-\tau_k)^{-1} \beta_k +w_\ell \P_\ell = \sum_{k=0}^\ell w_k\P_k = \P.
$$
\end{proof}

\subsection{Strict Equivalence\label{sec:strictLagrange}}
We prove the following theorem, establishing the strict equivalence of the companion pencil of equation~\eqref{eq:LagrangePencil}, for a matrix polynomial $\P(z)$ determined to grade $\ell$ by interpolation at $\ell+1$ points, to the monomial basis linearization for $\P(z) = 0\cdot z^{\ell+2} + 0\cdot z^{\ell+1} + \sum_{k=0}^{\ell} \A_k z^k$ considered as a grade $\ell + 2$ matrix polynomial.  This establishes that the Lagrange basis pencil is in fact a linearization, indeed a strong linearization, independently of the singularity or not of any of the values of the matrix polynomial, and independently of the regularity of the matrix pencil.
\begin{theorem}
Provided that the nodes $\tau_k$ are distinct, then the Lagrange basis linearization in equation~\eqref{eq:LagrangePencil}, namely $\L_L(z) = z\C_{L,1} - \C_{L,0}$, is strictly equivalent to the monomial basis linearization $\L_M(z) = z\C_{M,1} - \C_{M,0}$; in other words, there exist nonsingular constant matrices $\U$ and $\W$ such that both $\U \C_{L,1}\W = \C_{M,1}$ and $\U \C_{L,0}\W = \C_{M,0}$. Explicitly, if $\V$ is the Vandermonde matrix with $(i,j)$ entry $v_{i,j} = \tau_{j-1}^{\ell+1-i}$ for $1 \le i,j \le \ell+1$, then we have
\begin{equation}
    \U = \begin{bmatrix}
    \I_n & 0 \\
    0 & \V\otimes \I_n
    \end{bmatrix}
\end{equation}
and, if the node polynomial $w(z) = \prod_{k=0}^\ell(z-\tau_k) = z^{\ell+1} + q_\ell z^\ell + \cdots + q_0$ has coefficients $q_k$, which we place in a row vector $\mat{q} = [q_\ell, q_{\ell-1}, \ldots, q_0]$,
\begin{equation}
    \W = \begin{bmatrix}
    \I_n & \mat{q} \\
    0 & \V^{-1} \otimes \I_n
    \end{bmatrix}\>.
\end{equation}
\end{theorem}
\begin{proof}
Notice first that $\det\U = \det\V\otimes\I_n = \prod_{i<j}(\tau_j-\tau_i)^n$ is not zero, and that $\det\W$ is the reciprocal of that.  Both these matrices are therefore nonsingular if the nodes are distinct.

Next, it is straightforward to verify that $\U \C_{L,1} \W = \C_{M,1}$ by multiplying out, and using the fact that the upper left corner block of each is the zero block, and that the rest is the identity because $(\V\otimes \I_n)\cdot(\V^{-1}\otimes\I_n) = \I_{n(\ell+1)}$.

The remainder of the proof consists of detailed examination of the consequences of multiplying out the more complicated $\U\C_{L,0}\W$.  From now on we drop the $\otimes\I_n$ notation as clutter, and the proof is considered only for the scalar case, but the tensor products should be kept in mind as the operations are read.  Writing
\begin{equation}
    \C_{L,0} = \begin{bmatrix}
    0 & -\P \\
    \beta & \D
    \end{bmatrix}
\end{equation}
where $\D = \diag(\tau_\ell, \tau_{\ell-1}, \ldots, \tau_0)$, we have (using $\P\V^{-1} = \A$)
\begin{equation}
    \C_{L,0}\W = \begin{bmatrix}
    0 & -\A \\
    \beta & \beta \mat{q} + \D\V^{-1}
    \end{bmatrix}\>.
\end{equation}
We seek to establish that 
\begin{equation}
    \U^{-1}\C_{M,0} = \begin{bmatrix}
    1 & \\
      & \V^{-1}
    \end{bmatrix}
    \begin{bmatrix}
    0 & -\A_\ell &         &\cdots & \A_0 \\
    1 & 0       &        &      &\\
      & 1       & 0      &      &\\
      &         & 1      & \ddots & \\
      &&& 1 & 0
    \end{bmatrix}
\end{equation}
is the same matrix.  This means establishing that 
\begin{equation}
    \begin{bmatrix}
    \beta & \beta\mat{q} + \D\V^{-1}
    \end{bmatrix}
    = \begin{bmatrix}
    \V^{-1} & 0
    \end{bmatrix}\>.\label{eq:leftshift}
\end{equation}

Begin by relating the monomial basis to the Lagrange basis:
\begin{equation}
    \begin{bmatrix}
    z^\ell \\
    z^{\ell-1} \\
    \vdots\\
    z\\
    1
    \end{bmatrix} = 
    \begin{bmatrix}
    \tau_\ell^\ell & \tau_{\ell-1}^\ell & \cdots & \tau_0^\ell \\
    \tau_{\ell}^{\ell-1} & \tau_{\ell-1}^{\ell-1} & \cdots & \tau_0^{\ell-1} \\
    \vdots & & & \vdots \\
    \tau_{\ell} & \tau_{\ell-1} & & \tau_0\\
    1 & 1 & \cdots & 1 
    \end{bmatrix}
    \begin{bmatrix}
    w_{\ell}(z) \\
    w_{\ell-1}(z)\\
    \vdots\\
    w_1(z)\\
    w_0(z)
    \end{bmatrix}\>,
    \label{eq:LagrangeChangeBasis}
\end{equation}
where the $w_k(z)$ are the Lagrange basis polynomials
$w_k(z) = \beta_k \prod_{j\ne k}(z-\tau_j)$.  The matrix with the powers of $\tau_k$ is, of course, just $\V$.  


Name the columns of $\V^{-1} = [\Lambda_\ell, \Lambda_{\ell-1}, \ldots, \Lambda_0 ]$.
We thus have to show that the final column of equation~\eqref{eq:leftshift} is $0$ and that
\begin{equation}
    \Lambda_{j-1} = \beta q_j + \D\Lambda_j \label{eq:vectorversion}
\end{equation}
for $j=1$, $2$, $\ldots$, $\ell$, and that $\Lambda_\ell = \beta$.  
But this is just a translation into matrix algebra of the $\ell+2$ polynomial coefficients of
\begin{equation}
    \beta_i w(z) = (z-\tau_i)w_j(z)\qquad 0 \le i \le \ell\>,
\end{equation}
or $zw_i(z) = \beta_i w(z) + \tau_i w_i(z)$.  The coefficients of different powers of $z$ in these identities provide the columns in the matrix equation~\eqref{eq:leftshift}.
This completes the proof.
\end{proof}
\section{Concluding Remarks}
This paper shows how to use scalar matrix tools---specifically the Maple code for computing the Hermite normal form---and computation with low-dimensional examples to solve problems posed for \emph{matrix polynomials} of arbitrary dimension and arbitrary block size.  Current symbolic computation systems do not seem to offer facilities for direct computation with such objects.

The mathematical problems that we examined included the explicit construction of unimodular matrix polynomial cofactors $\E(z)$ and $\F(z)$ which would demonstrate that a given linearization $\L(z)$ for a matrix polynomial $\P(z)$ was, indeed, a linearization.  

We showed that the approach  discovered cofactors that were valid \emph{generically}.  This is because the primary tool used here, namely the Hermite normal form, is discontinuous at special values of the parameters (and indeed discovery of those places of discontinuity is the main purpose of the Hermite and Smith normal forms).  Nonetheless, in the case of the monomial basis and others we were able to guess such a universal form (by replacing a leading coefficient block by the identity block) from the HNF.  A separate investigation, based on the change-of-basis matrix but also inspired by experimental computation, found new explicit universal cofactors for all bases considered here.  We also introduced a new reversal for polynomials expressed in the Bernstein basis which may have better numerical properties than the standard reversal does.

For the Bernstein basis, which is symmetric under $z \to 1-z$, we were initially puzzled that the Hermite analogue had a problem if $\P(1)$ was singular but not when $\P(0)$ was singular.  This asymmetry disappears by considering instead a Hermite analogue of the form
\begin{equation}
    \H = \begin{bmatrix}
    \P(z) & & & & \\
    \H_{\ell-2} & \I_n & & & \\
    \vdots & & \ddots & & \\
    \H_0 & & & & \I_n
    \end{bmatrix}\>.
\end{equation}
This works only if $\P(0)^{-1}$ exists.

We plan to apply this method to Hermite interpolational bases, where (as previously for Lagrange interpolational bases) only local linearizations are currently known in the literature.

\begin{acks}
RMC thanks Froil\'an Dopico for several very useful discussions on linearization, giving several references, and pointing out the difference between local linearization and linearization. The support of Western University's New International Research Network grant is also gratefully acknowledged.
\end{acks}

\bibliographystyle{plainnat}
\bibliography{references}

\end{document}